\documentclass[12pt, reqno]{amsart}
\usepackage{amscd,amsmath,amsthm,amssymb,graphics}
\usepackage{amsfonts,amssymb,amscd,amsmath,enumitem,verbatim}
\usepackage[a4paper,top=3cm,left=3cm,right=3cm]{geometry}
\theoremstyle{plain}
\usepackage{hyperref}
\newtheorem{Theorem}{Theorem}
\newtheorem{Lemma}[Theorem]{Lemma}

\newtheorem{Proposition}[Theorem]{Proposition}

\theoremstyle{definition}
\newtheorem{Definition}[Theorem]{Definition}
\newtheorem{Remark}[Theorem]{Remark}

\usepackage{float}
\begin{document}
\title{Some notes on the algebraic structure of linear recurrent sequences}
\author {Gessica Alecci$^1$, Stefano Barbero$^1$, Nadir Murru$^2$
}

\address{$^1$ Department of Mathematical Sciences G. L. Lagrange, Politecnico di Torino}
\email{gessica.alecci@polito.it, stefano.barbero@polito.it}
\address{$^2$ Department of Mathematics, Università di Trento}
\email{nadir.murru@unitn.it}
\date{}

\begin{abstract}
Several operations can be defined on the set of all linear recurrent sequences, such as the binomial convolution (Hurwitz product) or the multinomial convolution (Newton product). Using elementary techniques, we prove that this set equipped with the termwise sum and the aforementioned products are $R-$algebras, given any commutative ring $R$ with identity. Moreover, we provide explicitly a characteristic polynomial of the Hurwitz product and Newton product of any two linear recurrent sequences. Finally, we also investigate whether these $R-$algebras are isomorphic, considering also the $R-$algebras obtained using the Hadamard product and the convolution product.
\end{abstract}

\maketitle
\section{Introduction}
Given a commutative ring with identity $R$, we will denote by $\mathcal S(R)$ the set of all sequences $\mathbf a = (a_n)_{n \geq 0}$ such that $a_n \in R$, for all $n \in \mathbb N$. A sequence $\mathbf a \in \mathcal S(R)$ is said to be a \emph{linear recurrent sequence} with characteristic polynomial $p_a(t) = t^N - \sum_{i=0}^{N-1} h_i t^{N-i}$ if its elements satisfy the following relation
\[a_n = \sum_{i=0}^{N-1} h_i a_{n-i}\]
for all $n \geq N$ and the elements $a_0, \ldots a_{N-1}$ are called \emph{initial conditions}. We will denote by $\mathcal W(R) \subset \mathcal S(R)$ the set of all linear recurrent sequences. Moreover, given $\mathbf a \in \mathcal S(R)$, we will write $A_o(t)=\sum_{n=0}^\infty a_nt^n$ for the ordinary generating function (o.g.f.) and $A_e(t)=\sum_{n=0}^\infty \cfrac{a_n}{n!}t^n$ for the exponential generating function (e.g.f.). It is well known that $\mathcal S(R)$ and $\mathcal W(R)$ can be equipped with several operations giving them interesting algebraic structures. 
When $R$ is a field, it is immediate to see that the element-wise sum or product (also called the Hadamard product) of two linear recurrent sequences is still a linear recurrent sequence, see, e.g., \cite{RSbook}. In \cite{CV}, the authors proved it in the more general case where $R$ is a ring, showing that $\mathcal W(R)$ is an $R-$algebra and giving also explicitly the characteristic polynomials of the element-wise sum and Hadamard product of two linear recurrent sequences. Larson and Taft \cite{LT, Taft} studied this algebraic structure characterizing the invertible elements and zero divisors. Further studies about the behaviour of linear recurrent sequences under the Hadamard product can be found, e.g., in \cite{Cak, GN, Kaz, Zierler}.  
Similarly, $\mathcal W(R)$ equipped with the element-wise sum and the convolution product (or Cauchy product) has been deeply studied. For instance,  $\mathcal W(R)$ is still an $R-$algebra and the characteristic polynomial of the convolution product between two linear recurrent sequences can be explicitly found \cite{CV}. The convolution product of linear recurrent sequences is very important in many applications and it has been studied also from a combinatorial point of view \cite{ABCM} and over finite fields \cite{Hau}. For other results, see, e. g., \cite{Stoll, Sza, Sza2}. 
Another important operation between sequences is the binomial convolution (or Hurwitz product). The Hurwitz series ring, introduced in a systematic way by Keigher \cite{Keigher}, has been extensively studied by several authors \cite{BCM2, BCM3, Ben, Ben2, Keigher2, Liu}. However, there are few results when focusing on linear recurrent sequences \cite{Kur, Kur2}. 

In this paper, we extend the studies about the algebraic structure of linear recurrent sequences considering in particular the Hurwitz product and the Newton product (which is the generalization of the Hurwitz product considering multinomial coefficients). In particular, we prove that $\mathcal W(R)$ is an $R-$algebra when equipped with the element-wise sum and the Hurwitz product, as well as when we consider element-wise sum and Newton product. We also give explicitly the characteristic polynomials of the Hurwitz and Newton product of two linear recurrent sequences. For the Newton product we find explicitly also the inverses. Moreover, we study the isomorphisms between these algebraic structures, finding that $\mathcal W(R)$ with element-wise sum and Hurwitz product is not isomorphic to the other algebraic structures, whereas if we consider the Newton product, there is an isomorphism with the $R-$algebra obtained using the Hadamard product. Finally, we provide an overview about the behaviour of linear recurrent sequences under all the different operations considered (element-wise sum, Hadamard product, Cauchy product, Hurwitz product, Newton product) with respect to the characteristic polynomials and their companion matrices.

\section{Preliminaries and notation}

For any $\mathbf a, \mathbf b \in \mathcal S(R)$, we will deal with the following operations:
\begin{itemize}
    \item \textit{componentwise sum} $+$, defined by
    \[ \mathbf a + \mathbf b = \mathbf c, \quad c_n = a_n + b_n, \quad \forall n \geq 0; \]
    \item \textit{componentwise product} or \textit{Hadamard product} $\odot$, defined by
    \[ \mathbf a \odot \mathbf b = \mathbf c, \quad c_n = a_n \cdot b_n, \quad \forall n \geq 0; \]
    \item \textit{convolution product} $*$, defined by
    \[ \mathbf a * \mathbf b = \mathbf c, \quad c_n = \sum_{i=0}^n a_i b_{n-i}, \quad \forall n \geq 0; \]
    \item \textit{binomial convolution product} or \textit{Hurwitz product} $\star$, defined by
    \[ \mathbf a \star \mathbf b = \mathbf c, \quad c_n =  \sum_{i=0}^n \binom{n}{i}a_i b_{n-i}, \quad \forall n \geq 0; \]
    \item \textit{multinomial convolution product} or \textit{Newton product} $\boxtimes$, defined by
    \[ \mathbf a \boxtimes \mathbf b = \mathbf c, \quad c_n =  \sum_{i=0}^n \sum_{j=0}^i \binom{n}{i} \binom{i}{j} a_i b_{n-j}, \quad \forall n \geq 0. \]
\end{itemize}

\begin{Remark}
The Newton product is also called multinomial convolution product, since it is the natural generalization of the binomial convolution product using the multinomial coefficient, observing that $\binom{n}{i}\binom{i}{j} = \binom{n}{n-i, i-j, j}$.
\end{Remark}

In \cite{CV}, the authors showed that $(\mathcal W(R), +, \odot)$ and $(\mathcal W(R), +, *)$ are $R-$algebras and they are never isomorphic. Moreover, given $\mathbf a, \mathbf b \in \mathcal W(R)$ and $\mathbf c = \mathbf a \odot \mathbf b$, $\mathbf d = \mathbf a * \mathbf b$, they proved that 
\begin{equation} \label{eq:polcd}
p_c(t) = p_a(t) \otimes p_b(t), \quad p_d(t) = p_a(t) \cdot p_b(t),
\end{equation}
where the operation $\otimes$ between polynomials is defined as follows. Given two polynomials $f(t)$ and $g(t)$ with coefficients in $R$, said $F$ and $G$ their companion matrices, respectively, then $f(t) \otimes g(t)$ is the characteristic polynomial of the Kronecker product between $F$ and $G$. In the following, we will denote by $\otimes$ also the Kronecker product between matrices. To the best of our knowledge, similar results involving the Hurwitz product and the Newton product are still missing.

\begin{Remark}
Let us observe that the sequences $\mathbf c$ and $\mathbf d$, defined above, recur with characteristic polynomials $p_c(t)$ and $p_d(t)$ as given in \eqref{eq:polcd}, respectively, but these polynomials are not necessarily the minimal polynomials of recurrence. Indeed, it is an hard problem to find the minimal polynomials of recurrence of these sequences, for some results, see \cite{Cak, GN, LT, Stoll}.
\end{Remark}

\begin{Lemma} \label{Lemma3}
Given $\mathbf a \in \mathcal S(R)$, we have that $\mathbf a \in \mathcal W(R)$ if and only if $p^*_a(t) \cdot A_o(t)$ is a polynomial of degree less than $\deg(p_a(t))$, where $p^*_a(t)$ denotes the reciprocal or reflected polynomial of $p_a(t)$. 
\end{Lemma}
\begin{proof}
See \cite{CV}.
\end{proof}

\begin{Definition}
Given two monic polynomials $f(t)$ and $g(t)$ of degree $M$ and $N$, respectively, their resultant is $\text{res}(f(t), g(t)) := \prod_{i=1}^M \prod_{j=1}^N(\alpha_i - \beta_j)$, where $\alpha_i$'s and $\beta_j$'s are the roots of $f(t)$ and $g(t)$, respectively.
\end{Definition}

\section{R-algebras of linear recurrent sequences}

\begin{Theorem} \label{TeoHurwitz}
Given $\mathbf a, \mathbf b \in \mathcal W(R)$, we have that $\mathbf r = \mathbf a \star \mathbf b \in \mathcal W(R)$ and the characteristic polynomial of $\mathbf r$ is $\text{res}(p_a(x),p_b(t-x))$ with $p_b(t-x)$ regarded as a polynomial in t. Moreover, $(\mathcal W(R), +, \star )$ is an $R-$algebra.
\end{Theorem}

\begin{proof}
It is well--known that $(\mathcal S(R), +, \star)$ is an $R-$algebra (see, e.g., \cite{Keigher}), thus it is sufficient to show that $\mathcal W(R)$ is closed under the Hurwitz product for proving that $(\mathcal W(R), +, \star )$ is an $R-$algebra.

Let $M$ and $N$ be the degrees of $p_a(t)$ and $p_b(t)$, respectively. Let us suppose $p_a(t)$ and $p_b(t)$ have distinct roots denoted by $\alpha_1, \ldots, \alpha_M$ and $\beta_1, \ldots, \beta_N$, respectively. We consider the ordinary generating function of the sequence $\mathbf r = \mathbf a \star \mathbf b$,
\begin{align}\label{eq:primaformula}
  R_o(t)=  \sum_{n=0}^{+\infty} \left( \sum_{i=0}^{n}\binom{n}{i}a_ib_{n-i}\right)t^n &=\sum_{i=0}^{+\infty}  \sum_{n=i}^{+\infty}\binom{n}{i}a_ib_{n-i} t^n \nonumber \\
  &=\sum_{i=0}^{+\infty}a_it^i\sum_{n=i}^{+\infty}\binom{n}{i}b_{n-i} t^{n-i} \nonumber \\
  &= \sum_{i=0}^{+\infty}a_it^i\sum_{m=0}^{+\infty}\binom{m+i}{i}b_{m} t^{m}
\end{align}
where $\sum_{m=0}^{+\infty}\binom{m+i}{i}b_{m} t^{m}$ is the Hadamard product between the sequence $\mathbf b$ and $\left( \binom{m+i}{i} \right)_{m \geq 0}$, i.e., $\left(\sum_{m=0}^{+\infty}b_{m} t^m\right) \odot \left(\sum_{m=0}^{+\infty}\binom{m+i}{i}t^{m}\right)= B_o(t) \odot \cfrac{1}{{(1-t)}^{i+1}}$.
Since $B_o(t)$ is the ordinary generating function of the linear recurrent sequence $\mathbf b$, it is a rational function and we can write it as 
\[B_o(t)=\frac{\gamma(t)}{p^*_b(t)}=\sum_{j=1}^{N}\frac{c_j}{(1-\beta_jt)},\]
for some integers $c_j$'s. Now, we have
\begin{equation*}
    \frac{1}{1-\beta_jt} \odot \frac{1}{(1-t)^{i+1}}= \sum_{m=0}^{+\infty}(b_{j} t)^m \odot \sum_{m=0}^{+\infty}\binom{m+i}{i} t^{m} \\
    = \sum_{m=0}^{+\infty}\binom{m+i}{i}(b_{j} t)^{m} \\
    = \frac{1}{(1-\beta_jt)^{i+1}}
\end{equation*}
and we get that 
$$B_o(t) \odot \cfrac{1}{(1-t)^{i+1}}= \sum_{j=1} ^{N} c_j \cfrac{1}{1-\beta_jt} \odot \cfrac{1}{(1-t)^{i+1}}=\sum_{j=1} ^{N} \cfrac{c_j}{(1-\beta_jt)^{i+1}}.$$ 
Thus, from \eqref{eq:primaformula} we obtain
\begin{align}\label{eq:secondaformula}
R_o(t) = \sum_{i=0}^{+\infty}a_it^i\sum_{j=1} ^{N} \frac{c_j}{(1-\beta_jt)^{i+1}} &= \sum_{j=1} ^{N} \frac{c_j}{1-\beta_jt} \sum_{i=0}^{+\infty}a_i \frac{t^i}{(1-\beta_jt)^{i}} \nonumber \\
&= \sum_{j=1} ^{N} \frac{c_j}{1-\beta_jt} A_o \left( \frac{t}{1-\beta_jt}\right) \nonumber \\
&= \sum_{j=1} ^{N} \frac{c_j}{1-\beta_jt} \frac{\delta \left( \frac{t}{1-\beta_jt}\right)}{p^*_a \left( \frac{t}{1-\beta_jt}\right)} 
\end{align}
Let $p(t)=\text{res}(p_a(x), p_b(t-x))$, then ${\displaystyle p(t)= \prod_{h=1}^{M} \prod_{l=1}^{N} (t-\alpha_h -\beta_l} )$ and its reciprocal polynomial is ${\displaystyle p^*(t)= \prod_{h=1}^{M} \prod_{l=1}^{N} (1-(\alpha_h + \beta_l)t)}$. In particular, it is possible to rearrange the last formula in the following way
\begin{align}\label{eq:reciprocalres}
  p^*(t)= \prod_{h=1}^{M} \prod_{l=1}^{N} (1-\beta_lt -\alpha_ht)  &= \prod_{h=1}^{M} \prod_{l=1}^{N} (1-\beta_lt) \left (1-\frac{\alpha_ht}{1-\beta_lt} \right) \nonumber \\
  &= \prod_{h=1}^{M} \prod_{l=1}^{N} (1-\beta_lt) \prod_{l=1}^{N} \left (1-\frac{\alpha_ht}{1-\beta_lt} \right) \nonumber \\
  &= \prod_{h=1}^{M} p^*_b(t) \prod_{l=1}^{N} \left (1-\frac{\alpha_ht}{1-\beta_lt} \right) \nonumber \\
    &= \left [p^*_b(t) \right ]^M  \prod_{l=1}^{N} \prod_{h=1}^{M} \left (1-\frac{\alpha_ht}{1-\beta_lt} \right) \nonumber \\
    &= \left [p^*_b(t) \right ]^M  \prod_{l=1}^{N} p^*_a\left (\frac{t}{1-\beta_lt} \right)
\end{align}
Combining \eqref{eq:secondaformula} and \eqref{eq:reciprocalres} we get
\begin{align}\label{eq:3}
p^*(t) \cdot R_o(t) &= \left [p^*_b(t) \right ]^M  \prod_{l=1}^{N} p^*_a\left (\frac{t}{1-\beta_lt} \right) \left (  \sum_{j=1} ^{N} \frac{c_j}{(1-\beta_jt)} \frac{\delta \left( \frac{t}{1-\beta_jt}\right)}{p^*_a \left( \frac{t}{1-\beta_jt}\right)} \right) \nonumber \\
  &= \left [p^*_b(t) \right ]^M  \left ( \sum_{j=1} ^{N} \frac{c_j}{1-\beta_jt} \cdot \delta \left( \frac{t}{1-\beta_jt}\right) \right )  \cdot \prod_{\substack{ l=1 \\ l\neq j} }^{N} p^*_a\left (\frac{t}{1-\beta_lt} \right)  
\end{align}
Moreover the function $\delta \left( \frac{t}{1-\beta_jt}\right)$ can be written in the following way 
\begin{equation*}\label{eq:delta}
\delta \left( \frac{t}{1-\beta_jt}\right) = \sum_{h=0}^{M-1} \delta_h \cdot \left( \frac{t}{1-\beta_j t}\right )^h \nonumber \\
= \frac{\sum_{h=0}^{M-1} \delta_h t^h ( 1-\beta_j t )^{M-1-h}}{( 1-\beta_j t )^{M-1}} \nonumber \\
= \frac{\mu_j(t)}{( 1-\beta_j t )^{M-1}}
\end{equation*}
with $deg(\mu_j(t)) \leq M-1$. Applying the same reasoning,
\begin{equation}\label{eq:funcf}
p^*_a\left( \frac{t}{1-\beta_jt}\right)
= \frac{\sum_{h=0}^{M} f_h t^h ( 1-\beta_j t )^{M-h}}{( 1-\beta_j t )^{M}}  \nonumber \\
= \frac{\xi_j(t)}{( 1-\beta_j t )^{M}}
\end{equation}
with $deg(\xi_j(t)) \leq M$.
Hence, equation \eqref{eq:3} becomes
\begin{align}\label{eq:funcf2}
& \left [p^*_b(t) \right ]^M \sum_{j=1} ^{N} \frac{c_j}{1-\beta_jt} \cdot \frac{\mu_j(t)}{( 1-\beta_j t )^{M-1}}   \cdot \prod_{\substack{ l=1 \\ l\neq j} }^{N} \frac{\xi_j(t)}{( 1-\beta_j t )^{M}} \nonumber \\
&= \left [p^*_b(t) \right ]^M  \sum_{j=1} ^{N} c_j\cdot \frac{\mu_j(t)}{( 1-\beta_j t )^{M}} \frac{{\displaystyle \prod_{\substack{ l=1 \\ l\neq j} }^{N} \xi_j(t)}}{{\displaystyle \prod_{\substack{ l=1 \\ l\neq j} }^{N} ( 1-\beta_j t )^{M}}}    \nonumber \\
&= \left [p^*_b(t) \right ]^M  \sum_{j=1} ^{N} c_j \mu_j(t) \frac{{\displaystyle \prod_{\substack{ l=1 \\ l\neq j} }^{N} \xi_j(t)}}{\left [p^*_b(t) \right ]^M}   =\sum_{j=1} ^{N} c_j \mu_j(t) {\displaystyle \prod_{\substack{ l=1 \\ l\neq j} }^{N} \xi_j(t)}.
\end{align}
Since the degree of the polynomial \eqref{eq:funcf} is less than or equal to $MN-1$, by Lemma \ref{Lemma3}, then $\mathbf r$ is a linear recurrent sequence and $p(t)=\text{res}(p_a(x), p_b(t-x))$ is its characteristic polynomial, as desired.
\end{proof}

\begin{Remark} \label{rem:hurwitz}
Given $\mathbf a, \mathbf b \in \mathcal W(R)$, if $\alpha_1, \ldots \alpha_M$ and $\beta_1, \ldots, \beta_N$ are distinct roots of $p_a(t)$ and $p_b(t)$ respectively, then, by Theorem \ref{TeoHurwitz}, the roots of the characteristic polynomial of $\mathbf a \star \mathbf b$ are $\alpha_i + \beta_j$, for any $i = 1, \ldots, M$ and $j = 1, \ldots, N$. Moreover, we would like to highlight that the proof of Theorem \ref{TeoHurwitz} can be adapted also in the case of multiple roots.
\end{Remark}

\begin{Proposition} \label{def2newton}
Given $\mathbf a, \mathbf b \in \mathcal W(R)$, then $\mathbf a \boxtimes \mathbf b = [(\mathbf a \star \mathbf 1) \odot (\mathbf b \star \mathbf 1)] \star \mathbf{e}$, where $\mathbf 1 := (1, 1, 1, \ldots)$ and $\mathbf{e} := ((-1)^n)_{n \geq 0}$.
\end{Proposition}
\begin{proof}
The $n$-th terms of $\mathbf a \star \mathbf 1$ and $\mathbf b \star \mathbf 1$ are by definition $\sum_{i=0}^n \binom{n}{i}a_i$ and $\sum_{i=0}^n \binom{n}{i}b_i$, respectively. Thus, the $n$--th term of $(\mathbf a \star \mathbf 1) \odot (\mathbf b \star \mathbf 1)$ is $\sum_{s=0}^n \binom{n}{s}a_s \sum_{t=0}^n \binom{n}{t}b_t$, and thus $\sum_{i=0}^n \binom{n}{i}(-1)^{n-i}\sum_{s=0}^i \binom{i}{s}a_s \sum_{t=0}^i \binom{i}{t}b_t$ is the $n$--th term of $[(\mathbf a \star \mathbf 1) \odot (\mathbf b \star \mathbf 1)] \star \mathbf{e}$. \\
From the definition of Newton product, we want to prove the following equality
\begin{equation*}
 \sum_{i=0}^n \binom{n}{i}(-1)^{n-i}\sum_{s=0}^i \binom{i}{s}a_s \sum_{t=0}^i \binom{i}{t}b_t = \sum_{i=0}^n \sum_{j=0}^i \binom{n}{i} \binom{i}{j} a_i b_{n-j}, \quad \forall n \geq 0. 
\end{equation*}

Let $c_i= \sum_{s=0}^i \binom{i}{s}a_s \sum_{t=0}^i \binom{i}{t}b_t$ and $d_n= \sum_{i=0}^n \sum_{j=0}^i \binom{n}{i}\binom{i}{j}a_ib_{n-j}$, then the previous identity is equivalent to
\begin{equation}\label{uguaglianzaNewton2}
     \sum_{i=0}^n \binom{n}{i}  (-1)^ic_i=(-1)^n d_n.
\end{equation}

Exploiting the Newton's inversion formula, i.e., 
$$\sum_{i=0}^n \binom{n}{i}  (-1)^i f(i) = g(n) \Leftrightarrow f(n)=\sum_{i=0}^n \binom{n}{i}  (-1)^i g(i),$$
for some arithmetic functions $f$ and $g$, equation \eqref{uguaglianzaNewton2} becomes
\begin{equation*}
 \sum_{i=0}^n \binom{n}{i}  (-1)^i (-1)^i d_i = c_n   
\end{equation*}
 that is
\begin{equation*}
    \sum_{i=0}^n \binom{n}{i} \sum_{k=0}^i \sum_{j=0}^k \binom{i}{k}\binom{k}{j}a_kb_{i-j} = \sum_{s=0}^n \binom{n}{s}a_s \sum_{t=0}^n \binom{n}{t}b_t.
\end{equation*}
Now, we can write the first member as
\begin{align*}
 \sum_{j=0}^n \sum_{s=j}^n \sum_{i=s}^n \binom{n}{i}  \binom{i}{s}  \binom{s}{j}  a_sb_{i-j} &= \sum_{s=0}^n \sum_{j=0}^s \sum_{t=s-j}^{n-j} \binom{n}{t+j}  \binom{t+j}{s}  \binom{s}{j}  a_sb_t \nonumber \\
    &= \sum_{s=0}^n \binom{n}{s} a_s \sum_{j=0}^s \sum_{t=s-j}^{n-j}   \binom{n-s}{n-t-j}  \binom{s}{j}  b_t \nonumber \\
    &= \sum_{s=0}^n \binom{n}{s} a_s \sum_{m=0}^s \sum_{t=m}^{n-s+m}   \binom{n-s}{n-s-(t-m)}  \binom{s}{s-m}  b_t \nonumber \\
    &= \sum_{s=0}^n \binom{n}{s} a_s \sum_{m=0}^s \sum_{t=m}^{n-s+m}   \binom{n-s}{t-m}  \binom{s}{m}  b_t \nonumber \\
 &= \sum_{s=0}^n \binom{n}{s} a_s \sum_{t=0}^n b_t \sum_{m=0}^{t}   \binom{n-s}{t-m}  \binom{s}{m}   \nonumber \\
    &= \sum_{s=0}^n \binom{n}{s} a_s \sum_{t=0}^n \binom{n}{t} b_t \nonumber \\
\end{align*}
where the last equivalence is due to the Vandermonde's identity $\sum_{m=0}^{t}   \binom{n-s}{t-m}  \binom{s}{m} = \sum_{t=0}^n \binom{n}{t}$.
\end{proof}

\begin{Remark}
The previous proposition can be proved also exploiting the umbral calculus (see \cite{Rota} for the basic notions). Given $\mathbf a, \mathbf b \in \mathcal S(R)$, let us consider two linear functionals $U$ and $V$ defined by $U(W^n) = a_n$ and $V(Z^n) = b_n$. The $n$--th term of $(\mathbf a \boxtimes \mathbf b) \star \mathbf 1$ is $\sum_{i=0}^n \binom{n}{i}\sum_{k=0}^i\sum_{j=0}^k \binom{i}{k}\binom{k}{j}a_kb_{i-j}$ and, applying the functionals $U$ and $V$, it becomes
\[UV\left(\sum_{i=0}^n \binom{n}{i} Z^i\sum_{k=0}^i \binom{i}{k}W^k\sum_{j=0}^k\binom{k}{j}Z^{-j}\right) = UV\left( \sum_{i=0}^n \binom{n}{i} Z^i \sum_{k=0}^i \binom{i}{k} (W+W/Z)^k \right)\]
\[= UV\left(\sum_{i=0}^n \binom{n}{i}(ZW+W+Z)^i\right) = UV((ZW+W+Z+1)^n) \]
Now, the last quantity can be rewritten as
\[ UV((Z+1)^n(W+1)^n) = U(V(Z+1)^n(W+1)^n) = \]
\[ = U\left(V\left(\sum_{s=0}^n\binom{n}{s}Z^s\right)(W+1)^n\right) = U\left( \sum_{s=0}^n \binom{n}{s} b_s (W+1)^n \right) = \]
\[= \sum_{s=0}^n\binom{n}{s}b_s U\left( \sum_{t=0}^n \binom{n}{t} W^t \right) = \sum_{s=0}^n\binom{n}{s}b_s \cdot \sum_{t=0}^n\binom{n}{s}a_t,\]
which is the $n$--th term of the sequence $(\mathbf a \star \mathbf 1) \odot (\mathbf b \star \mathbf 1)$.
\end{Remark}

\begin{Theorem}
Given $\mathbf a, \mathbf b \in \mathcal W(R)$, we have that $\mathbf c = \mathbf a \boxtimes \mathbf b \in \mathcal W(R)$ and the characteristic polynomial of $\mathbf c$ is $\prod_{i=1}^M \prod_{j=1}^N(t-(\alpha_i+\beta_j+\alpha_i\beta_j))$, where $M = \deg(p_a(t))$, $N = \deg(p_b(t))$, $\alpha_i$'s are the roots of $p_a(t)$ and $\beta_j$'s the roots of $p_b(t)$. Moreover, $(\mathcal W(R), +, \boxtimes )$ is an $R-$algebra.
\end{Theorem}
\begin{proof}
Firstly, we show that $(\mathcal S(R), +, \boxtimes)$ is an $R-$algebra. This is an immediate consequence of Proposition \ref{def2newton}. Indeed, since $\mathbf a \boxtimes \mathbf b = [(\mathbf a \star \mathbf 1) \odot (\mathbf b \star \mathbf 1)] \star \mathbf{e}$, it is straightforward to see that $\boxtimes$ satisfies all the properties in order that $(\mathcal S(R), +, \boxtimes)$ is an $R-$algebra. 
Moreover, we can also see that $(1, 0, 0, \ldots)$ is the identity element for the Newton product. Indeed, it is sufficient to observe that $(1, 0, 0, \ldots)$ is the identity element for the Hurwitz product and $\mathbf{e}$ is the inverse of $\mathbf 1$ with respect to $\star$.
Then, it is immediate that $(\mathcal W(R), +, \boxtimes)$ is also an $R-$algebra, since given $\mathbf a, \mathbf b \in \mathcal W(R)$, we have $\mathbf a \boxtimes \mathbf b \in \mathcal W(R)$ by Proposition \ref{def2newton}.

By Theorem \ref{TeoHurwitz} and Remark \ref{rem:hurwitz}, we can observe that given $\mathbf a, \mathbf b \in \mathcal W(R)$, then $\mathbf a \star \mathbf 1$ and $\mathbf b \star \mathbf 1$ are linear recurrent sequences whose characteristic polynomials have roots $\alpha_i + 1$ and $\beta_j +1$, for $i = 1, \ldots, M$ and $j = 1, \ldots, N$, respectively. Moreover, since $\mathbf e$ is a linear recurrent sequence whose characteristic polynomial is $t + 1$, then $[(\mathbf a \star \mathbf 1) \odot (\mathbf b \star \mathbf 1)] \star \mathbf{e}$ has characteristic polynomial whose roots are $(\alpha_i+1)(\beta_j + 1) - 1 = \alpha_i + \beta_j + \alpha_i\beta_j$, for $i = 1, \ldots, M$ and $j = 1, \ldots, N$. 
\end{proof}

\begin{Proposition}\label{inversinewton}
Given $\mathbf a \in \mathcal S(R)$, said $\mathbf b$ its inverse with respect to the Newton product, then
\begin{equation}\label{formuladefinvnewton}
    b_n = (-1)^n \sum_{t=0}^n \binom{n}{t} (-1)^t \frac{1}{\sum_{s=0}^t \binom{t}{s}a_s}
\end{equation}
for any $n \geq 0$.
\begin{proof} \label{def1inversinewton}
Remembering that the identity element for the Newton product is $(1, 0, 0, \ldots)$, we have that $a_0b_0$ must be $1$, i.e., $b_0 = a_0^{-1}$. When $n \ge 1$, we have that
 \begin{equation*}
   \sum_{i=0}^n \binom{n}{i} (-1)^{n-i} \sum_{s=0}^i \binom{i}{s} a_s \sum_{t=0}^i \binom{i}{t} b_t = 0
 \end{equation*}
and,  then
\begin{equation*}
    \sum_{s=0}^i \binom{i}{s} a_s \sum_{t=0}^i \binom{i}{t} b_t = 1,  \quad
\sum_{t=0}^i \binom{i}{t} b_t = \frac{1}{\sum_{s=0}^i \binom{i}{s} a_s}. 
\end{equation*}
Let $d_i= \frac{1}{\sum_{s=0}^i \binom{i}{s} a_s}$, applying Newton's inversion formula, we get
\begin{equation*}
 d_i =\sum_{t=0}^i \binom{i}{t} (-1)^t (-1)^t b_t, \quad
(-1)^i b_i  = \sum_{t=0}^i \binom{i}{t} (-1)^t d_t 
\end{equation*}
and finally
\begin{equation*}
b_i  = (-1)^i \sum_{t=0}^i \binom{i}{t} (-1)^t \frac{1}{\sum_{s=0}^t \binom{t}{s} a_s} 
\end{equation*}
from which the thesis follows.
\end{proof}
\end{Proposition}

\begin{Remark}
Let us observe that $\mathbf a$ is invertible with respect to the Newton product if and only if all the elements of $\mathbf a$ are invertible elements of $R$, as well as it happens for the Hadamard product.
\end{Remark}

Let us observe that every $\mathbf a \in \mathcal W(R)$ can be associated to its monic characteristic polynomial $p_a(t)$ with coefficients in $R$ and this polynomial to its companion matrix $A$. As studied the 
$R-$algebras of kind $(\mathcal W(R), +, \odot)$, $(\mathcal W(R), +, *)$, $(\mathcal W(R), +, \star)$ and $(\mathcal W(R), +, \boxtimes)$, it is interesting to give to the set of the monic polynomials $\mathcal Pol(R)$ with coefficients in $R$ some new algebraic structures. Moreover, we can also observe what happens to the roots and to the companion matrices of the characteristic polynomials.

Let us consider $\mathbf a, \mathbf b \in \mathcal W(R)$ with characteristic polynomials of degree $M$ and $N$, whose roots are $\alpha_1, \ldots, \alpha_M$ and $\beta_1, \ldots \beta_N$, respectively.
The sequences $\mathbf a + \mathbf b$ and $\mathbf a * \mathbf b$ both recur with characteristic polynomial $p_a(t) \cdot p_b(t)$.

Regarding the Hadamard product, we have already observed that the characteristic polynomial of $\mathbf c = \mathbf a \odot \mathbf b$ is $p_c(t) = p_a(t) \otimes p_b(t)$, whose roots are $\alpha_i\beta_j$, for $i = 1, \ldots, M$ and $j = 1, \ldots, N$. Thus, starting from the $R-$algebra $(\mathcal W(R), +, \odot)$, we can construct the semiring $(\mathcal Pol(R), \cdot, \otimes)$ with identity the polynomial $t-1$. 
Said $A, B,$ and $C$ the companion matrices of $p_a(t), p_b(t)$ and $p_c(t)$, we have that $C = A \otimes B$, where $\otimes$ is the Kronecker product between matrices. Thus $C$ is a $mn \times mn$ matrix with eigenvalues the products of the eigenvalues of $A$ and $B$.

Similarly, starting from the Hurwitz product, we can construct a new operation in $\mathcal Pol(R)$. Given $\mathbf c = \mathbf a \star \mathbf b$, we proved that $p_c(t)$ has roots $\alpha_i + \beta_j$, for $i = 1, \ldots, M$ and $j= 1, \ldots, N$. The matrix $A \otimes I_n + I_m \otimes B$ is a $mn \times mn$ matrix, whose eigenvalues are the sum of the eigenvalues of $A$ and $B$. Thus, we can define $p_c(t) = p_a(t) \star p_b(t)$ as the characteristic polynomial of the matrix $A \otimes I_n + I_m \otimes B$ and we get the semiring $(\mathcal Pol(R), \cdot, \star)$.

Finally, given $\mathbf c = \mathbf a \boxtimes \mathbf b$, we know that $p_c(t)$ has roots $\alpha_i + \beta_j + \alpha_i\beta_j$, for $i = 1, \ldots, M$ and $j = 1, \ldots, N$. In this case, we can define $p_c(t) = p_a(t) \boxtimes p_b(t)$ as the characteristic polynomial of the matrix $A \otimes I_n + I_m \otimes B + A \otimes B$, which is a $mn \times mn$ matrix, whose eigenvalues are exactly $\alpha_i + \beta_j + \alpha_i\beta_j$, for $i = 1, \ldots, M$ and $j = 1, \ldots, N$. Thus, we have that $(\mathcal Pol(R), \cdot, \boxtimes)$ is another semiring of monic polynomials.

\section{On isomorphisms between R-algebras}

In \cite{CV}, the authors proved that $(\mathcal W(R), +, \odot)$ and $(\mathcal W(R), +, *)$ are never isomorphic as $R-$algebras. In the following we prove similar results for the other algebraic structures that we have studied in the previous section.

\begin{Theorem}
The $R-$algebras $(\mathcal W(R), +, \odot)$ and $(\mathcal W(R), +, \star)$ are not isomorphic.
\end{Theorem}
\begin{proof}

Let us suppose that $\psi:(\mathcal W(R), +, \odot) \xrightarrow{}(\mathcal W(R), +, \star)$ is an injective morphism and consider $\mathbf a:=(1, 0, 0, 0, \dots)$ and $\mathbf b:=(0,1, 0, 0, \dots)$.\\
Then $\psi (\mathbf a \odot \mathbf b ) = \psi (\mathbf a ) \star \psi (\mathbf b )= (0, 0, \ldots)$ and, by injectivity, $\psi (\mathbf a ), \psi (\mathbf b ) \ne (0,0, \ldots)$. Let $A_e^{\psi}(t)$ and $B_e^{\psi}(t)$ be the exponential generating functions of $\psi (\mathbf a)$ and $\psi (\mathbf b)$, respectively.  From $\psi (\mathbf a ) \star \psi (\mathbf b )=(0,0, \ldots)$, it follows that $A_e^{\psi}(t)B_e^{\psi}(t)=0$. From Lemma \ref{Lemma3}, we have
\begin{equation}\label{eq:8}
    p^*_{\psi(\mathbf a)}(t)A^{\psi}_o(t)=h(t)
\end{equation}
with $deg(h(t)) < deg(p^*_{\psi(\mathbf a)}(t)) $. \\
Now, consider the map $\gamma: \sum_{n=0}^\infty a_nt^n \xrightarrow{} \sum_{n=0}^\infty \frac{a_n}{n!}t^n$, which is an isomorphism between the ring of ordinary series and the Hurwitz series ring, see, e.g., \cite{BCM3}.
Applying $\gamma$ to \eqref{eq:8}, we obtain
\begin{align*}
     \gamma \left (p^*_{\psi(\mathbf a)}(t)\right) A^{\psi}_e(t) &=\gamma (h(t)) \nonumber\\
\end{align*}
where $p^*_{\psi(\mathbf a)}(t)$ can be viewed as a formal series with an infinite number of zero coefficients.
Multiplying by $\gamma (p^*_{\psi(\mathbf a)}(t))$ the equation $A_e^{\psi}(t)B_e^{\psi}(t)=0$, it becomes
\begin{align*}
   \gamma (h(t))B_e^{\psi}(t)&=0
\end{align*}
which implies $h(t)B_o^{\psi}(t)=0$. From this, it follows that there is a nonzero element $w \in R$, such that $wB_o^{\psi}(t)=0$ (\cite[Eq. (2.9)]{GIL}) and $w \mathbf b =0$, which is absurd.
\end{proof}

\begin{Theorem}
The $R-$algebras $(\mathcal W(R), +, \odot)$ and $(\mathcal W(R), +, \boxtimes)$ are isomorphic.
\end{Theorem}
\begin{proof}
The explicit isomorphism is $\psi: (\mathcal W(R), +, \odot) \rightarrow (\mathcal W(R), +, \boxtimes)$ defined by $\psi(\mathbf a) := \mathbf a \star \mathbf e$, where $\mathbf e = ((-1)^n)_{n \geq 0}$. Indeed, by Theorem \ref{TeoHurwitz}, the map $\psi$ is well--defined (in the sense that a linear recurrent sequence is mapped into a linear recurrent sequence). Moreover, since $\mathbf 1$ is the inverse of $\mathbf e$ with respect to the Hurwitz product, it is straightforward to check injectivity and surjectivity. Finally, by Proposition \ref{def2newton}, we have
\[\psi(\mathbf a) \boxtimes \psi(\mathbf b) = (((\mathbf a \star \mathbf e) \star \mathbf 1) \odot ((\mathbf b \star \mathbf e) \star \mathbf 1) )\star \mathbf e\]
that is equal to
\[\psi(\mathbf a \odot \mathbf b) = (\mathbf a \odot \mathbf b) \star \mathbf e\]
since $\mathbf e \star \mathbf 1 = (1, 0, 0, \ldots)$.
\end{proof}

\begin{Theorem}
Let $R$ be an integral domain, if $\psi: (\mathcal W(R), +, *) \rightarrow (\mathcal W(R), +, \star)$  is a morphism, then $\psi$ is not injective.
\end{Theorem}
\begin{proof}
Let us suppose that $\psi:(\mathcal W(R), +, *) \rightarrow (\mathcal W(R), +, \star)$ is an injective morphism. Let us denote by $(\psi(\mathbf a))_i$ the $i$-th term of the sequence $\psi(\mathbf a)$. The $n$-th term of $\psi(\mathbf a) \star \psi(\mathbf b)$ is
\begin{equation*}
    \sum_{i=0}^n \binom{n}{i}(\psi(\mathbf{a}))_i(\psi(\mathbf{b}))_{n-i}= n!\sum_{i=0}^n \binom{n}{i}\frac{(\psi(\mathbf{a}))_i}{i!}\frac{(\psi(\mathbf{b}))_{n-i}}{(n-i)!}.
\end{equation*}
Then, considering $\psi(\mathbf a * \mathbf b) = \psi(\mathbf a) \star \psi(\mathbf b)$, for any $\mathbf a, \mathbf b \in \mathcal W(R)$, we obtain
\begin{equation} \label{riscr}
    \psi(\mathbf{a}*\mathbf{b})=((\psi(\mathbf{a}) \odot \mathbf{f}^{-1})*(\psi(\mathbf{b}) \odot \mathbf{f}^{-1})) \odot \mathbf{f}
\end{equation}
where we define the formal sequences $\mathbf{f}:=(1, 2!, 3!, \ldots)$ and $\mathbf{f}^{-1}:=(1, \frac{1}{2!}, \frac{1}{3!}, \ldots)$.

We define a map $\tau:(\mathcal W(R), +, *) \xrightarrow{}(\mathcal W(R), +, *)$ such that $\tau(\mathbf a) = \psi(\mathbf a) \odot \mathbf f^{-1}$. By definition, we have that $\tau( \mathbf{a}+ \mathbf{b} )= \tau( \mathbf{a})+\tau( \mathbf{b} )$,  $\tau( \mathbf{a} * \mathbf{b} )= \tau( \mathbf{a})*\tau( \mathbf{b} )$, for any $\mathbf a, \mathbf b \in \mathcal W(R)$ and $\ker (\tau)= \mathbf{0}$ where $\mathbf{0}= (0, 0, 0, \ldots)$.

Let $\Tilde{\tau}$ be a map that acts over the ordinary generating functions such that if $A(t)=\sum_{n=0}^{+\infty} a_nt^n$ then $\Tilde{\tau}(A(t))=\sum_{n=0}^{+\infty}(\tau(a))_nt^n$. From the properties of $\tau$, it follows that $\Tilde{\tau}(A(t)+B(t))=\Tilde{\tau}(A(t))+\Tilde{\tau}(B(t))$, $\Tilde{\tau}(A(t)B(t))=\Tilde{\tau}(A(t))\Tilde{\tau}(B(t))$ and $\Tilde{\tau}(A(t))=\Tilde{\tau}(B(t)) \Leftrightarrow A(t)=B(t)$. 

Now, if we consider $A(t)=1$ and $B(t)=1$, we clearly have $\Tilde{\tau}(1)=\Tilde{\tau}(1 \cdot 1)=\Tilde{\tau}(1) \Tilde{\tau}(1)$ and this implies $\Tilde{\tau}(1)=1$. Indeed, by the injectivity of $\tau$, we can not have $\Tilde{\tau}(1)=0$ because $\tau(1, 0, \ldots)$ should be $\mathbf{0}$. 

In the case that $A(t)=t$ and $B(t)=-t$, then  $0= \Tilde{\tau}(0)=\Tilde{\tau}(t -t)=\Tilde{\tau}(t) +\Tilde{\tau}(-t)$, so $\Tilde{\tau}(-t)=-\Tilde{\tau}(t)$. Moreover, when $A(t)=t$ and $B(t)=t^{-1}$, then  $1= \Tilde{\tau}(1)=\Tilde{\tau}(t \cdot t^{-1})=\Tilde{\tau}(t)\Tilde{\tau}(t^{-1})$ and $\Tilde{\tau}(t^{-1})=( \Tilde{\tau}(t))^{-1}$. 

Lastly, if $A(t)=B(t)=t$, then  $\Tilde{\tau}(t^2)=\Tilde{\tau}(t \cdot t)=\Tilde{\tau}(t)\Tilde{\tau}(t)$, so $\Tilde{\tau}(t^2)=( \Tilde{\tau}(t))^2$ and $\Tilde{\tau}(nt)=n\Tilde{\tau}(t)$, $\forall n \in \mathbb{N}$.

Let us consider $\Tilde{\tau}(t)=\sum_{n=0}^{+\infty} s_nt^n$, then from $\Tilde{\tau}(t^2)=( \Tilde{\tau}(t))^2$ it follows that 
\begin{equation}\label{condizioni}
    \sum_{n=0}^{+\infty} s_nt^n= \sum_{n=0}^{+\infty} \left(\sum_{i=0}^n s_is_{n-i}\right)t^n \Rightarrow
    \begin{cases}
    \sum_{i=0}^{2k+1} s_is_{2k+1-i}=0 \\
    \sum_{i=0}^{2k} s_is_{2k-i}=s_k
    \end{cases}
\end{equation}
From \eqref{condizioni}, we obtain $s_0=s_0^2$ and we may have $s_0=0$ or $s_0=1$. 
In the case that $s_0=1$, then $s_i=0$, $\forall i \ge 1$, and $\Tilde{\tau}(t)=t$.
Whereas, if $s_0=0$, then $s_1=s_1^2$ and $s_1=0$ or $s_1=1$. In the case that $s_1=1$, then $s_i=0$, $\forall i \ge 2$, and $\Tilde{\tau}(t)=t$.\\
In the case that $s_1=0$, then $s_2=s_2^2$ so $s_2=0$ or $s_2=1$. Repeating the same reasoning and exploiting \eqref{condizioni}, we get in general that $\Tilde{\tau}(t)=t^k$ must hold for a fixed $k \ge 1$. 

Let us consider $A(t)=\frac{1}{1-t}$, then 
$$\Tilde{\tau}((1-t)^{-1})=(\Tilde{\tau}(1)-\Tilde{\tau}(t))^{-1}=(1-\Tilde{\tau}(t))^{-1}=(1-t^k)^{-1}.$$
By definition, $\Tilde{\tau} (\frac{1}{1-t})=\Tilde{\tau}(\sum_{n=0}^{+\infty} t^n)= \sum_{n=0}^{+\infty} (\tau (\mathbf{1}))_nt^n$ and $\Tilde{\tau} (\frac{1}{1-t})=\sum_{n=0}^{+\infty} t^{kn}$. Comparing the coefficients, we have 
\[ \left( \Tilde{\tau} \left( \frac{1}{1-t} \right) \right) _n = \begin{cases}  1 \quad \text{if $k \mid n$} \cr 0 \quad \text{if $k \nmid n$} \end{cases}.\]

Thus, by \eqref{riscr} and the definition of $\tau$ and $\Tilde{\tau}$, we have
\[ \psi(\mathbf 1)_n = (\tau(\mathbf 1) \odot \mathbf f)_n = \begin{cases}  n! \quad \text{if $k \mid n$} \cr 0 \quad \text{if $k \nmid n$} \end{cases}  \]
which is not a linear recurrent sequence.
\end{proof}

\section*{Conflict of interest}
The authors assert that there are no conflicts of interest.

\end{document}